\documentclass[11pt,twoside]{article}
\def\co{\tiny{\textcircled{\tiny\#}}}
\usepackage{mathrsfs}
\usepackage{amsmath}
\usepackage{amsthm}
\usepackage{amsfonts}
\usepackage{amssymb}
\usepackage{cite}
\usepackage{color}
\usepackage{latexsym}
\usepackage[all]{xy}
\usepackage{extarrows}

\date{\empty}
\allowdisplaybreaks[4] 

\numberwithin{equation}{section} \theoremstyle{plain}
\newtheorem*{thm*}{Main Theorem}
\newtheorem{theorem}{Theorem}[section]
\newtheorem{corollary}[theorem]{Corollary}
\newtheorem*{corollary*}{Corollary}

\newtheorem*{claim*}{Claim}
\newtheorem{lemma}[theorem]{Lemma}
\newtheorem*{lemma*}{Lemma}

\newtheorem*{proposition*}{Proposition}

\newtheorem*{remark*}{Remark}
\newtheorem{example}[theorem]{Example}
\newtheorem*{example*}{Example}

\newtheorem*{question*}{Question}

\newtheorem*{definition*}{Definition}

\newtheorem*{acknowledgements*}{ACKNOWLEDGEMENTS}

\begin{document}
\begin{center}
{\large \bf  Characterizations of weighted core inverse in rings with involution}

\vspace{0.4cm} {\small \bf Tingting Li}
\footnote{ Tingting Li (Corresponding author E-mail: littnanjing@163.com): School of Mathematical Sciences, Yangzhou University, Yangzhou, 225002, China.}

\end{center}

 \bigskip
 { \bf  Abstract:}  \leftskip0truemm\rightskip0truemm
$R$ is a unital ring with involution.
We investigate the characterizations and representations of weighted core inverse of an element in $R$ by idempotents and units.
For example,
let $a\in R$ and $e\in R$ be an invertible Hermitian element,
$n\geqslant 1$,
then $a$ is $e$-core invertible if and only if there exists an element (or an idempotent) $p$ such that $(ep)^{\ast}=ep$,
$pa=0$ and $a^{n}+p$ (or $a^{n}(1-p)+p$) is invertible.
As a consequence,
let $e, f\in R$ be two invertible Hermitian elements,
then $a$ is weighted-$\mathrm{EP}$ with respect to $(e, f)$ if and only if there exists an element (or an idempotent) $p$ such that $(ep)^{\ast}=ep$,
$(fp)^{\ast}=fp$,
$pa=ap=0$ and $a^{n}+p$ (or $a^{n}(1-p)+p$) is invertible.
These results generalize and improve conclusions in \cite{Li}.


 \bigskip

\section{ \bf Introduction}
Let $R$ be a unital ring.
An involution $\ast$ in $R$ is an anti-isomorphism of degree 2 in $R$,
that is to say,
$(a^{\ast})^{\ast}=a, (a+b)^{\ast}=a^{\ast}+b^{\ast}$ and $(ab)^{\ast}=b^{\ast}a^{\ast}$ for all $a, b\in R$.
An element $a\in R$ is Hermitian if and only if $a^{\ast}=a$.
And $a\in R$ is called an idempotent if $a^{2}=a$.
A Hermitian idempotent is said to be a projection.

Let $e,f\in R$ be two invertible Hermitian elements.
We say that $a\in R$ has a weighted Moore-Penrose inverse with weights $e, f$ if there exists an $x$ satisfying the following four equations (see for example, \cite{KPS2, Wei, Mosic}):
\begin{equation*}
\begin{split}
    (1)~axa=a,~(2)~xax=x,~(3e)~(eax)^{\ast}=ax,~(4f)~(fxa)^{\ast}=xa.
\end{split}
\end{equation*}
If such an $x$ exists,
then it is called a weighted Moore-Penrose inverse of $a$,
it is unique and denoted by $a^{\dagger}_{e, f}$.
If $e=f=1$,
then $a^{\dagger}_{e, f}=a^{\dagger}$ is the ordinary Moore-Penrose inverse of $a$.
The sets of all Moore-Penrose invertible elements and weighted Moore-Penrose invertible elements with weights $e, f$ are denoted by the symbols $R^{\dagger}$ and $R^{\dagger}_{e, f}$,
respectively.

If $x$ satisfies the equations $(1)$ and $(3e)$,
then $x$ is called a $\{1,3e\}$-inverse of $a$ and denoted by $a^{(1,3e)}$,
and a $\{1,4f\}$-inverse of $a$ can be similarly defined.
The symbols $a\{1,3e\}$ and $a\{1,4f\}$ denote the sets of all $\{1,3e\}$-inverses and $\{1,4f\}$-inverses of $a$,
respectively.
An element $a\in R$ is group invertible if there is an $x\in R$ satifying
$$(1)~axa=a,~~~(2)~xax=x,~~~(5)~ax=xa.$$
Such an $x$ is the group inverse of $a$,
and it is unique if it exists and denoted by $a^{\#}$.
The sets of all group invertible,
invertible,
$\{1,3e\}$-invertible and $\{1,4f\}$-invertible elements in $R$ are denoted by the symbols $R^{\#}$,
$R^{-1}$,
$R^{(1,3e)}$ and $R^{(1,4f)}$,
respectively.

Mosi\'{c} et al. \cite{Mosic2} give introduced the definitions of the $e$-core inverse and the $f$-dual core inverse,
which extended the concepts of the core inverse and the dual core inverse (see \cite{OM, DSR, XSZ}).
The element $a\in R$ has a weighted core inverse with weight $e$ (or $e$-core inverse),
if there exists $x\in R$ such that
$$axa = a, ~~~xR = aR, ~~~Rx = Ra^{\ast}e.$$
If such an $x$ exists,
then it is called a $e$-core inverse of $a$,
it is unique and denoted by $a^{e, \co}$.
The element $a\in R$ has a weighted dual core inverse with weight $f$ (or $f$-dual core inverse),
if there exists $x\in R$ such that
$$axa = a, ~~~fxR = a^{\ast}R, ~~~Rx = Ra.$$
If such an $x$ exists,
then it is called a $f$-core inverse of $a$,
it is unique and denoted by $a_{f, \co}$.
If $e=1$ (resp., $f=1$),
then $a^{e, \co}=a^{\co}$ (resp., $a_{f, \co}=a_{\co}$) is the ordinary core inverse (resp., dual core inverse) of $a$.
The sets of all $e$-core invertible elements and $f$-dual core invertible elements in $R$ are denoted by the symbols $R^{e, \co}$ and $R_{f, \co}$,
respectively.
In addition,
they defined that an element $a\in R$ is weighted-$\mathrm{EP}$ with respect to $(e, f)$ if $a\in R^{\#}\cap R^{\dagger}_{e, f}$ and $a^{\#}=a^{\dagger}_{e, f}$.
When $e=f=1$,
it is the ordinary EP element.

In this paper,
we investigate the characterizations of $e$-core inverse,
$f$-dual core inverse,
right $e$-core inverse and weighted-$\mathrm{EP}$ with respect to $(e, f)$ in a ring with involution.
These results will generalize and improve conclusions in \cite{Li}.

We will also use the following notations:
$aR=\{ax ~|~ x\in R\}$, $Ra=\{xa ~|~ x\in R\}$, $^{\circ}\!a=\{x\in R ~|~ xa=0\}$, $a^{\circ}=\{x\in R ~|~ ax=0\}$.

Some auxiliary lemmas and results are presented for the further reference.

\begin{lemma} \cite[Proposition $7$]{H}\label{group-inverse}
Let $a\in R$.
$a\in R^{\#}$ if and only if $a=a^{2}x=ya^{2}$ for some $x, y\in R$.
In this case, $a^{\#}=yax=y^{2}a=ax^{2}$.
\end{lemma}

\begin{lemma} \cite{BAB, DSC}\label{J.L}
Let $a, b\in R$.
$1+ab$ is invertible if and only if $1+ba$ is invertible.
\end{lemma}

\begin{lemma} \cite[Proposition $2.1$ and $2.2$]{HHZ}\label{13e-14f-inverse}
Let $a\in R$ and let $e\in R$ be an invertible Hermitian element.
We have the following results:\\
$(i)$ $a$ is $\{1,3e\}$-invertible if and only if $a\in Ra^{\ast}ea$.
Moreover,
if $a=xa^{\ast}ea$ for some $x\in R$,
then $x^{\ast}e\in a\{1, 3e\}$;\\
$(ii)$ $a$ is $\{1,4f\}$-invertible if and only if $a\in af^{\-1}a^{\ast}R$.
Moreover,
if $a=af^{\-1}a^{\ast}y$ for some $y\in R$,
then $f^{-1}y^{\ast}\in a\{1, 4f\}$.
\end{lemma}

\begin{lemma} \cite[Theorem $2.1$ and $2.2$]{Mosic2}\label{e-core-inverse}
Let $a\in R$ and let $e, f\in R$ be two invertible Hermitian elements.
Then \\
$(I)$ The following statements are equivalent:\\
\indent $(i)$ $a\in R^{e, \co}$;\\
\indent $(ii)$ there exists $x\in R$ such that
$$(1)~axa=a,~(2)~xax=x,~(3e)~(eax)^{\ast}=eax,~(6)~xa^2=a,~(7)~ax^2=x;$$
\indent $(iii)$ there exists $x\in R$ such that
$$(3e)~(eax)^{\ast}=eax,~(6)~xa^2=a,~(7)~ax^2=x;$$
\indent $(iv)$ $a\in R^{\#}\cap R^{(1, 3e)}$.\\
\indent In this case,
$$a^{e, \co}=x=a^{\#}aa^{(1, 3e)}.$$
$(II)$ The following statements are equivalent:\\
\indent $(i)$ $a\in R_{f, \co}$;\\
\indent $(ii)$ there exists $x\in R$ such that
$$(1)~axa=a,~(2)~xax=x,~(4f)~(fxa)^{\ast}=fxa,~(8)~a^2x=a,~(9)~x^2a=x;$$
\indent $(iii)$ there exists $x\in R$ such that
$$(4f)~(fxa)^{\ast}=fxa,~(8)~a^2x=a,~(9)~x^2a=x;$$
\indent $(iv)$ $a\in R^{\#}\cap R^{(1, 4f)}$.\\
\indent In this case,
$$a_{f, \co}=x=a^{(1, 4f)}aa^{\#}.$$
\end{lemma}

\begin{lemma} \cite{WL2}\label{r-core}
Let $a\in R$ and let $e\in R$ be an invertible Hermitian element,
$n\geqslant 2$.
Then the following statements are equivalent:\\
$(i)$ $a\in Ra^{\ast}ea\cap a^nR$;\\
$(ii)$ $a\in R(a^{\ast})^nea$.
\end{lemma}

\section{Main results}\label{a}
In this section,
we present some new equivalent conditions for the existences of core inverses.

\begin{theorem}\label{thm1}
Let $a\in R$ and let $e\in R$ be an invertible Hermitian element,
$n\geqslant 2$.
We have the following results:\\
$(i)$ $a\in R^{e, \co}$ if and only if $a\in R(a^{\ast})^{n}ea\cap Ra^{n}$.
If $a=s(a^{\ast})^{n}ea$ for some $s\in R$,
then $a^{e, \co}=a^{n-1}s^{\ast}e$;\\
$(ii)$ $a\in R_{f, \co}$ if and only if $a\in af^{-1}(a^{\ast})^{n}R\cap a^{n}R$.
If $a=af^{-1}(a^{\ast})^{n}t$ for some $t\in R$,
then $a_{f, \co}=f^{-1}t^{\ast}a^{n-1}$.
\end{theorem}

\begin{proof}
$(i)$. According to Lemma~\ref{group-inverse},
it is easy to deduce that $a\in R^{\#}$ if and only if $a\in a^{n}R\cap Ra^{n}$ for $n\geqslant 2$.
From Lemma~\ref{r-core} we know that $a\in R(a^{\ast})^nea$ if and only if $a\in Ra^{\ast}ea\cap a^nR$,
thus

\begin{equation*}
\begin{split}
    a\in R^{e, \co}
    &~\xLongleftrightarrow{Lemma~\ref{e-core-inverse}} a\in R^{(1,3e)}\cap R^{\#}\\
    &~\xLongleftrightarrow{Lemma~\ref{13e-14f-inverse}} a\in Ra^{\ast}ea\cap R^{\#}\\
    &~\Longleftrightarrow a\in Ra^{\ast}ea\cap a^{n}R\cap Ra^{n}\\
    &~\Longleftrightarrow a\in R(a^{\ast})^{n}ea\cap Ra^{n}.
\end{split}
\end{equation*}

Next,
we give the representation of $a^{e, \co}$.
Since $a\in R(a^{\ast})^{n}ea\cap Ra^{n}$,
there exists $s\in R$ such that $a=s(a^{\ast})^{n}ea$.
By Lemma~\ref{13e-14f-inverse},
we have
\begin{equation*}
\begin{split}
    (s(a^{\ast})^{n-1})^{\ast}e=a^{n-1}s^{\ast}e\in a\{1,3e\}.
\end{split}
\end{equation*}
Using Lemma~\ref{e-core-inverse},
we obtain
\begin{equation*}
\begin{split}
    a^{e, \co}=a^{\#}aa^{(1,3e)}=a^{\#}a(a^{n-1}s^{\ast}e)=a^{n-1}s^{\ast}e.
\end{split}
\end{equation*}

$(ii)$. Similarly as $(i)$.
\end{proof}

From the above proof,
it is easy to see that Theorem~\ref{thm1} is also true in a semigroup.
Next we characterize the existence of $e$-core inverse by idempotents and units.

\begin{theorem} \label{thm2}
Let $a\in R$ and let $e\in R$ be an invertible Hermitian element,
$n\geqslant 1$ is a positive integer.
The following statements are equivalent: \\
$(i)$ $a\in R^{e, \co}$; \\
$(ii)$ there exists a unique idempotent $p$ such that $(ep)^{\ast}=ep$, $pa=0$, $u=a^{n}+p\in R^{-1}$;\\
$(iii)$ there exists an element $s$ such that $(es)^{\ast}=es$, $sa=0$, $v=a^{n}+s\in R^{-1}$;\\
$(iv)$ there exists a unique idempotent $q$ such that $(eq)^{\ast}=eq$, $qa=0$, $w=a^{n}(1-q)+q\in R^{-1}$;\\
$(v)$ there exists an element $t$ such that $(et)^{\ast}=et$, $ta=0$, $z=a^{n}(1-t)+t\in R^{-1}$.\\
In this case,
when $n=1$,
$$a^{e, \co}=u^{-1}(1-p)=v^{-1}av^{-1}=(1-q)w^{-1}=z^{-1}a(1-t)z^{-1},$$
when $n\geqslant 2$,
$$a^{e, \co}=a^{n-1}u^{-1}=a^{n-1}v^{-1}=a^{n-1}(1-q)w^{-1}=a^{n-1}(1-t)z^{-1}.$$
\end{theorem}

\begin{proof}
$(i)\Rightarrow (ii).$ Set $p=1-aa^{e, \co}$,
we observe first that $p$ is an idempotent satisfying $(ep)^{\ast}=ep$ and $pa=0$.
Next we show that $u=a^{n}+p=a^{n}+1-aa^{e, \co}\in R^{-1}$.

When $n=1$,
it is easy to verify
\begin{equation*}
\begin{split}
    (a+1-aa^{e, \co})(a^{e, \co}+1-a^{e, \co}a)=1=(a^{e, \co}+1-a^{e, \co}a)(a+1-aa^{e, \co}),
\end{split}
\end{equation*}
thus $u=a+1-aa^{e, \co}\in R^{-1}$.

When $n\geqslant 2$,
we prove the invertibility of $u$ by induction on $n$.
Since $1+aa^{e, \co}(a-1)=a+1-aa^{e, \co}\in R^{-1}$,
from Lemma~\ref{J.L} we obtain that $a^{2}a^{e, \co}+1-aa^{e, \co}=1+(a-1)aa^{e, \co}\in R^{-1}$.
Therefore,
$a^{2}+1-aa^{e, \co}=(a^{2}a^{e, \co}+1-aa^{e, \co})(a+1-aa^{e, \co})$ is invertible.
Assume that $n>2$ and $a^{n-1}+1-aa^{e, \co}\in R^{-1}$,
then $a^{n}+1-aa^{e, \co}=(a^{2}a^{e, \co}+1-aa^{e, \co})(a^{n-1}+1-aa^{e, \co})$ is invertible.

Finally,
we prove the uniqueness of the idempotent $p$.
It is necessary to show that $^{\circ}\!(a^{n})=$ $^{\circ}\!(1-p)$.
If $xa^{n}=0$,
then
\begin{equation*}
\begin{split}
    0=xa^{n}=x(1-p)a^{n}=x(1-p)(a^{n}+p).
\end{split}
\end{equation*}
Since $a^{n}+p\in R^{-1}$,
we have $x(1-p)=0$.
Conversely,
if $y(1-p)=0$,
then $ya^{n}=y(1-p)a^{n}=0$.
Thus $^{\circ}\!(a^{n})=$ $^{\circ}\!(1-p)$.
Assume that $p, p_1$ are both idempotents which satisfy the statement $(ii)$,
then $^{\circ}\!(1-p)=$ $^{\circ}\!(a^{n})=$ $^{\circ}\!(1-p_1)$.
By $p\in$ $^{\circ}\!(1-p)=$ $^{\circ}\!(1-p_1)$,
we obtain $p=pp_1$.
Similarly,
we can get $p_1=p_1p$ from $p_1\in$ $^{\circ}\!(1-p_1)=$ $^{\circ}\!(1-p)$.
Then
\begin{equation*}
\begin{split}
    ep=(ep)^{\ast}=(epp_1)^{\ast}=(epe^{-1}ep_1)^{\ast}=(ep_1)^{\ast}(e^{-1})^{\ast}(ep)^{\ast}=ep_1e^{-1}ep=ep_1p=ep_1,
\end{split}
\end{equation*}
thus $p=p_1$.

$(ii)\Rightarrow (i).$ Suppose that there exists a unique idempotent $p$ such that $(ep)^{\ast}=ep$, $pa=0$, $u=a^{n}+p\in R^{-1}$.

When $n=1$,
set $x=u^{-1}(1-p)$.
Now we prove that $x$ is the $e$-core inverse of $a$.
Since $(1-p)u=a$,
$ua^2=a$ and $pu=p$,
we obtain
\begin{equation}\label{2.1}
\begin{split}
    1-p=au^{-1},
\end{split}
\end{equation}
\begin{equation}\label{2.2}
\begin{split}
    a=u^{-1}a^2,
\end{split}
\end{equation}
\begin{equation}\label{2.3}
\begin{split}
    (1-p)u^{-1}=u^{-1}-p.
\end{split}
\end{equation}
Therefore,
\begin{equation*}
\begin{split}
    eax=eau^{-1}(1-p)\stackrel{(\ref{2.1})}{=}e(1-p)(1-p)=e-ep \text{~is Hermitian},
\end{split}
\end{equation*}
\begin{equation*}
\begin{split}
    xa^2=u^{-1}(1-p)a^2=u^{-1}a^2\stackrel{(\ref{2.2})}{=}a,
\end{split}
\end{equation*}
\begin{equation*}
\begin{split}
    ax^2
    &~=au^{-1}(1-p)u^{-1}(1-p)\stackrel{(\ref{2.1})}{=}(1-p)u^{-1}(1-p)\\
    &~\stackrel{(\ref{2.3})}{=}(u^{-1}-p)(1-p)=u^{-1}(1-p)=x.
\end{split}
\end{equation*}
As a consequence,
$a\in R^{e, \co}$ with $a^{e, \co}=u^{-1}(1-p)$ by Lemma~\ref{e-core-inverse}.

When $n\geqslant 2$,
from $u=a^{n}+p\in R^{-1}$ and $p^{\ast}=(e^{-1}ep)^{\ast}=epe^{-1}$,
we get $u^{\ast}=(a^{\ast})^{n}+epe^{-1}\in R^{-1}$.
Since $ua=a^{n+1}$ and $u^{\ast}ea=(a^{\ast})^{n}ea$,
\begin{equation*}
\begin{split}
    a=u^{-1}a^{n+1}\in Ra^{n}
\end{split}
\end{equation*}
and
\begin{equation*}
\begin{split}
    a=(u^{\ast}e)^{-1}(a^{\ast})^{n}ea\in R(a^{\ast})^{n}ea.
\end{split}
\end{equation*}
Thus $a\in R(a^{\ast})^{n}ea\cap Ra^{n}$,
and then $a\in R^{e, \co}$ and $a^{e, \co}=a^{n-1}(eu)^{-1}e=a^{n-1}u^{-1}$ by Theorem~\ref{thm1}.

$(ii)\Rightarrow (iii)$ is trivial.

$(iii)\Rightarrow (i).$ Assume that there exists an element $s$ such that $(es)^{\ast}=es$, $sa=0$, $v=a^{n}+s\in R^{-1}$.

When $n=1$.
Since $v=a^{n}+s\in R^{-1}$ and
\begin{equation}\label{2.4}
\begin{split}
    s^{\ast}=(e^{-1}es)^{\ast}=ese^{-1},
\end{split}
\end{equation}
we get $v^{\ast}=a^{\ast}+ese^{-1}\in R^{-1}$.
From $va=a^2$ and $v^{\ast}ea=a^{\ast}ea$ we obtain
\begin{equation}\label{2.5}
\begin{split}
    a=v^{-1}a^2\in Ra^2
\end{split}
\end{equation}
and
\begin{equation}\label{2.6}
\begin{split}
    a=(v^{\ast}e)^{-1}a^{\ast}ea\in Ra^{\ast}ea.
\end{split}
\end{equation}
It remains to show that $a\in a^2R$.
By Lemma~\ref{13e-14f-inverse},
the equation (\ref{2.6}) shows that $a\in R^{(1, 3e)}$ with $v^{-1}\in a\{1, 3e\}$.
Hence
\begin{equation}\label{2.7}
\begin{split}
    a=aa^{(1, 3e)}a=av^{-1}a.
\end{split}
\end{equation}
In addtion,
$sv=s^2$ leads to $s=s^2v^{-1}$,
thus
\begin{equation}\label{2.8}
\begin{split}
    es^2v^{-1}
    &~=es=(es)^{\ast}=(es^2v^{-1})^{\ast}=(v^{\ast})^{-1}s^{\ast}es\\
    &~\stackrel{(\ref{2.4})}{=}(v^{\ast})^{-1}ese^{-1}es=(v^{\ast})^{-1}es^2.
\end{split}
\end{equation}
Moreover,
\begin{equation}\label{2.9}
\begin{split}
    esv^{-2}
    &~\stackrel{(\ref{2.8})}{=}[(v^{\ast})^{-1}es^2]v^{-2} = (v^{\ast})^{-1}(es^2v^{-1})v^{-1}\\
    &~\stackrel{(\ref{2.8})}{=}(v^{\ast})^{-1}[(v^{\ast})^{-1}es^2]v^{-1}=(v^{\ast})^{-2}es^2v^{-1}\stackrel{(\ref{2.8})}{=}(v^{\ast})^{-2}es,
\end{split}
\end{equation}
thus
\begin{equation}\label{2.10}
\begin{split}
    ev^{-1}
    &~=evv^{-2}=e(a+s)v^{-2}=eav^{-2}+esv^{-2}\stackrel{(\ref{2.9})}{=}eav^{-2}+(v^{\ast})^{-2}es.
\end{split}
\end{equation}
Therefore,
\begin{equation}\label{2.11}
\begin{split}
    a
    &~\stackrel{(\ref{2.7})}{=}av^{-1}a=ae^{-1}(ev^{-1})a\stackrel{(\ref{2.10})}{=}ae^{-1}[eav^{-2}+(v^{\ast})^{-2}es]a=a^2v^{-2}a\in a^2R.
\end{split}
\end{equation}
According to Lemma~\ref{group-inverse},
equations (\ref{2.5}) and (\ref{2.11}) show that $a\in R^{\#}$ with $a^{\#}=v^{-2}a$.
And from (\ref{2.6}) and Lemma~\ref{e-core-inverse},
$a\in R^{e, \co}$ with
\begin{equation*}
\begin{split}
    a^{e, \co}=a^{\#}aa^{(1, 3e)}=v^{-2}aav^{-1}\stackrel{(\ref{2.5})}{=}v^{-1}av^{-1}.
\end{split}
\end{equation*}

When $n\geqslant 2$,
from $v=a^{n}+s\in R^{-1}$ and $s^{\ast}=ese^{-1}$,
we get $v^{\ast}=(a^{\ast})^{n}+ese^{-1}\in R^{-1}$.
Since $va=a^{n+1}$ and $v^{\ast}ea=(a^{\ast})^{n}ea$,
\begin{equation*}
\begin{split}
    a=v^{-1}a^{n+1}\in Ra^{n}
\end{split}
\end{equation*}
and
\begin{equation*}
\begin{split}
    a=(v^{\ast}e)^{-1}(a^{\ast})^{n}ea\in R(a^{\ast})^{n}ea.
\end{split}
\end{equation*}
Thus $a\in R(a^{\ast})^{n}ea\cap Ra^{n}$,
and then $a\in R^{e, \co}$ and $a^{e, \co}=a^{n-1}(ev)^{-1}e=a^{n-1}v^{-1}$ by Theorem~\ref{thm1}.

$(i)\Rightarrow (iv).$ Set $q=1-aa^{e, \co}$,
then $q$ is an idempotent satisfying $(eq)^{\ast}=eq$ and $qa=0$.
Now we prove that $w=a^{n}(1-q)+q=a^{n+1}a^{e, \co}+1-aa^{e, \co}\in R^{-1}$.

When $n=1$,
from the proof of $(i)\Rightarrow (ii)$ we know that $1+aa^{e, \co}(a-1)=a+1-aa^{e, \co}\in R^{-1}$,
thus $w=a^{2}a^{e, \co}+1-aa^{e, \co}=1+(a-1)aa^{e, \co}\in R^{-1}$ by the Lemma~\ref{J.L}.

When $n\geqslant 2$,
from the proof of $(i)\Rightarrow (ii)$ we know that $1+aa^{e, \co}(a^n-1)=a^n+1-aa^{e, \co}\in R^{-1}$,
thus $w=a^{n+1}a^{e, \co}+1-aa^{e, \co}=1+(a^n-1)aa^{e, \co}\in R^{-1}$ by the Lemma~\ref{J.L}.

In the end,
we prove the uniqueness of the idempotent $q$.
It is necessary to show that $^{\circ}\!(a^{n})=$ $^{\circ}\!(1-q)$.
If $xa^{n}=0$,
then
\begin{equation*}
\begin{split}
    0=xa^{n}(1-q)=x(1-q)[a^{n}(1-q)+q].
\end{split}
\end{equation*}
From $a^{n}(1-q)+q\in R^{-1}$ we have $x(1-q)=0$.
Conversely,
if $y(1-q)=0$,
then $ya^{n}=y(1-q)a^{n}=0$.
Hence $^{\circ}\!(a^{n})=$ $^{\circ}\!(1-q)$.
Suppose that $q, q_1$ are both idempotents which satisfy the statement $(iv)$,
then $^{\circ}\!(1-q)=$ $^{\circ}\!(a^{n})=$ $^{\circ}\!(1-q_1)$.
From $q\in$ $^{\circ}\!(1-q)=$ $^{\circ}\!(1-q_1)$ we obtain $q=qq_1$.
Similarly,
$q\in$ $^{\circ}\!(1-q_1)=$ $^{\circ}\!(1-q)$ leads to $q_1=q_1q$.
Then
\begin{equation*}
\begin{split}
    eq=(eq)^{\ast}=(eqq_1)^{\ast}=(eqe^{-1}eq_1)^{\ast}=(eq_1)^{\ast}(e^{-1})^{\ast}(eq)^{\ast}=eq_1e^{-1}eq=eq_1q=eq_1,
\end{split}
\end{equation*}
thus $q=q_1$.

$(iv)\Rightarrow (i).$ Assume that there exists a unique idempotent $q$ such that $(eq)^{\ast}=eq$, $qa=0$, $w=a^{n}(1-q)+q\in R^{-1}$.

When $n=1$,
let $x=(1-q)w^{-1}$.
Now we show that $x$ is the $e$-core inverse of $a$.
Since $(1-q)w=a(1-q)$ and $wa^2=a$,
we get
\begin{equation}\label{2.12}
\begin{split}
    1-q=a(1-q)w^{-1},
\end{split}
\end{equation}
\begin{equation}\label{2.13}
\begin{split}
    a=w^{-1}a^2.
\end{split}
\end{equation}
Therefore,
\begin{equation*}
\begin{split}
    eax=ea(1-q)w^{-1}\stackrel{(\ref{2.12})}{=}e(1-q)=e-eq \text{~is Hermitian},
\end{split}
\end{equation*}
\begin{equation*}
\begin{split}
    xa^2=(1-q)w^{-1}a^2\stackrel{(\ref{2.13})}{=}(1-q)a=a,
\end{split}
\end{equation*}
\begin{equation*}
\begin{split}
    ax^2
    &~=a(1-q)w^{-1}(1-q)w^{-1}\stackrel{(\ref{2.12})}{=}(1-q)w^{-1}=x.
\end{split}
\end{equation*}
As a consequence,
$a\in R^{e, \co}$ with $a^{e, \co}=(1-q)w^{-1}$ by Lemma~\ref{e-core-inverse}.

When $n\geqslant 2$,
from $w=a^{n}(1-q)+q\in R^{-1}$ and $q^{\ast}=(e^{-1}eq)^{\ast}=eqe^{-1}$,
we get $w^{\ast}=(1-p)^{\ast}(a^{\ast})^{n}+eqe^{-1}\in R^{-1}$.
Since $wa=a^{n+1}$ and $w^{\ast}ea=(1-q)^{\ast}(a^{\ast})^{n}ea$,
\begin{equation*}
\begin{split}
    a=w^{-1}a^{n+1}\in Ra^{n}
\end{split}
\end{equation*}
and
\begin{equation*}
\begin{split}
    a=(w^{\ast}e)^{-1}(1-q)^{\ast}(a^{\ast})^{n}ea\in R(a^{\ast})^{n}ea.
\end{split}
\end{equation*}
Therefore,
$a\in R(a^{\ast})^{n}ea\cap Ra^{n}$,
and then $a\in R^{e, \co}$ and $a^{e, \co}=a^{n-1}(1-q)(ew)^{-1}e=a^{n-1}(1-q)w^{-1}$ by Theorem~\ref{thm1}.

$(iv)\Rightarrow (v)$ is trivial.

$(v)\Rightarrow (i).$ Assume that there exists an element $t$ such that $(et)^{\ast}=et$, $ta=0$, $z=a^{n}(1-t)+t\in R^{-1}$.

When $n=1$.
Since $z=a^{n}(1-t)+t\in R^{-1}$ and
\begin{equation}\label{2.14}
\begin{split}
    t^{\ast}=(e^{-1}et)^{\ast}=ete^{-1},
\end{split}
\end{equation}
we get $z^{\ast}=(1-t)^{\ast}a^{\ast}+ese^{-1}\in R^{-1}$.
From $za=a^2$ and $z^{\ast}ea=(1-t)^{\ast}a^{\ast}ea$ we obtain
\begin{equation}\label{2.15}
\begin{split}
    a=z^{-1}a^2\in Ra^2
\end{split}
\end{equation}
and
\begin{equation}\label{2.16}
\begin{split}
    a=(z^{\ast}e)^{-1}(1-t)^{\ast}a^{\ast}ea\in Ra^{\ast}ea.
\end{split}
\end{equation}
It remains to show that $a\in a^2R$.
By Lemma~\ref{13e-14f-inverse},
the equation (\ref{2.16}) shows that $a\in R^{(1, 3e)}$ with $(1-t)z^{-1}\in a\{1, 3e\}$.
Hence
\begin{equation}\label{2.17}
\begin{split}
    a=aa^{(1, 3e)}a=a(1-t)z^{-1}a.
\end{split}
\end{equation}
In addtion,
$tz=t^2$ leads to $t=t^2z^{-1}$,
thus
\begin{equation}\label{2.18}
\begin{split}
    et^2z^{-1}
    &~=et=(et)^{\ast}=(et^2z^{-1})^{\ast}=(z^{\ast})^{-1}t^{\ast}et\\
    &~\stackrel{(\ref{2.14})}{=}(z^{\ast})^{-1}ete^{-1}et=(z^{\ast})^{-1}et^2.
\end{split}
\end{equation}
Moreover,
\begin{equation}\label{2.19}
\begin{split}
    etz^{-2}
    &~\stackrel{(\ref{2.18})}{=}[(z^{\ast})^{-1}et^2]z^{-2} = (z^{\ast})^{-1}(et^2z^{-1})z^{-1}\\
    &~\stackrel{(\ref{2.18})}{=}(z^{\ast})^{-1}[(z^{\ast})^{-1}et^2]z^{-1}=(z^{\ast})^{-2}et^2z^{-1}\stackrel{(\ref{2.18})}{=}(z^{\ast})^{-2}et,
\end{split}
\end{equation}
thus
\begin{equation}\label{2.20}
\begin{split}
    ez^{-1}
    &~=ezz^{-2}=e[a(1-t)+t]z^{-2}=ea(1-t)z^{-2}+etz^{-2}\\
    &~\stackrel{(\ref{2.19})}{=}ea(1-t)z^{-2}+(z^{\ast})^{-2}et.
\end{split}
\end{equation}
Therefore,
\begin{equation}\label{2.21}
\begin{split}
    a
    &~\stackrel{(\ref{2.17})}{=}a(1-t)z^{-1}a=a(1-t)e^{-1}(ez^{-1})a\\
    &~\stackrel{(\ref{2.20})}{=}a(1-t)e^{-1}[ea(1-t)z^{-2}+(z^{\ast})^{-2}et]a\\
    &~=a^2(1-t)z^{-2}a\in a^2R.
\end{split}
\end{equation}
According to Lemma~\ref{group-inverse},
equations (\ref{2.15}) and (\ref{2.21}) show that $a\in R^{\#}$ with $a^{\#}=z^{-2}a$.
And from (\ref{2.16}) and Lemma~\ref{e-core-inverse} deducing that $a\in R^{e, \co}$ with
\begin{equation*}
\begin{split}
    a^{e, \co}=a^{\#}aa^{(1, 3e)}=z^{-2}aa(1-t)z^{-1}\stackrel{(\ref{2.15})}{=}z^{-1}a(1-t)z^{-1}.
\end{split}
\end{equation*}

When $n\geqslant 2$,
from $z=a^{n}(1-t)+t\in R^{-1}$ and $t^{\ast}=ete^{-1}$ we get $z^{\ast}=(1-t)^{\ast}(a^{\ast})^{n}+ete^{-1}\in R^{-1}$.
Since $za=a^{n+1}$ and $z^{\ast}ea=(1-t)^{\ast}(a^{\ast})^{n}ea$,
\begin{equation*}
\begin{split}
    a=z^{-1}a^{n+1}\in Ra^{n}
\end{split}
\end{equation*}
and
\begin{equation*}
\begin{split}
    a=(z^{\ast}e)^{-1}(1-t)^{\ast}(a^{\ast})^{n}ea\in R(a^{\ast})^{n}ea.
\end{split}
\end{equation*}
Thus $a\in R(a^{\ast})^{n}ea\cap Ra^{n}$,
and then $a\in R^{e, \co}$ and $a^{e, \co}=a^{n-1}(1-t)(ev)^{-1}e=a^{n-1}(1-t)v^{-1}$ by Theorem~\ref{thm1}.
\end{proof}

When $e=1$,
we get the following result,
which improves Theorem $3.3$ and Theorem $3.4$ in \cite{Li}.

\begin{corollary}
Let $a\in R$,
$n\geqslant 1$ is a positive integer.
The following statements are equivalent: \\
$(i)$ $a\in R^{\co}$; \\
$(ii)$ there exists a unique projection $p$ such that $pa=0$, $u=a^{n}+p\in R^{-1}$;\\
$(iii)$ there exists a Hermitian element $s$ such that $sa=0$, $v=a^{n}+s\in R^{-1}$;\\
$(iv)$ there exists a unique projection $q$ such that $qa=0$, $w=a^{n}(1-q)+q\in R^{-1}$;\\
$(v)$ there exists a Hermitian element $t$ such that $ta=0$, $z=a^{n}(1-t)+t\in R^{-1}$.\\
In this case,
when $n=1$,
$$a^{\co}=u^{-1}(1-p)=v^{-1}av^{-1}=(1-q)w^{-1}=z^{-1}a(1-t)z^{-1},$$
when $n\geqslant 2$,
$$a^{\co}=a^{n-1}u^{-1}=a^{n-1}v^{-1}=a^{n-1}(1-q)w^{-1}=a^{n-1}(1-t)z^{-1}.$$
\end{corollary}

There is an analogous result for the $f$-dual core inverse of $a\in R$,
we omit its proof.

\begin{theorem} \label{thm3}
Let $a\in R$ and let $f\in R$ be an invertible Hermitian element,
$n\geqslant 1$ is a positive integer.
The following statements are equivalent: \\
$(i)$ $a\in R_{f, \co}$; \\
$(ii)$ there exists a unique idempotent $p$ such that $(fp)^{\ast}=fp$, $ap=0$, $u=a^{n}+p\in R^{-1}$;\\
$(iii)$ there exists an element $s$ such that $(fs)^{\ast}=fs$, $as=0$, $v=a^{n}+s\in R^{-1}$;\\
$(iv)$ there exists a unique idempotent $q$ such that $(fq)^{\ast}=fq$, $aq=0$, $w=(1-q)a^{n}+q\in R^{-1}$;\\
$(v)$ there exists an element $t$ such that $(ft)^{\ast}=ft$, $at=0$, $z=(1-t)a^{n}+t\in R^{-1}$.\\
In this case,
when $n=1$,
$$a_{f, \co}=(1-p)u^{-1}=v^{-1}av^{-1}=w^{-1}(1-q)=z^{-1}(1-t)az^{-1},$$
when $n\geqslant 2$,
$$a_{f, \co}=u^{-1}a^{n-1}=v^{-1}a^{n-1}=(1-q)a^{n-1}w^{-1}=(1-t)a^{n-1}z^{-1}.$$
\end{theorem}

When $f=1$,
we get the following result.

\begin{corollary}
Let $a\in R$,
$n\geqslant 1$ is a positive integer.
The following statements are equivalent: \\
$(i)$ $a\in R_{\co}$; \\
$(ii)$ there exists a unique projection $p$ such that $ap=0$, $u=a^{n}+p\in R^{-1}$;\\
$(iii)$ there exists a Hermitian element $s$ such that $as=0$, $v=a^{n}+s\in R^{-1}$;\\
$(iv)$ there exists a unique projection $q$ such that $aq=0$, $w=(1-q)a^{n}+q\in R^{-1}$;\\
$(v)$ there exists a Hermitian element $t$ such that $at=0$, $z=(1-t)a^{n}+t\in R^{-1}$.\\
In this case,
when $n=1$,
$$a_{\co}=(1-p)u^{-1}=v^{-1}av^{-1}=w^{-1}(1-q)=z^{-1}(1-t)az^{-1},$$
when $n\geqslant 2$,
$$a_{\co}=u^{-1}a^{n-1}=v^{-1}a^{n-1}=(1-q)a^{n-1}w^{-1}=(1-t)a^{n-1}z^{-1}.$$
\end{corollary}

In the derivation of $(i)$ from $(ii)$ of Theorem~\ref{thm2},
when $n=1$,
$a^{\co}=u^{-1}(1-p)$.
From the equation (\ref{2.1}) (i.e., $1-p=au^{-1}$) we know that $e(1-p)=eau^{-1}$ is Hermitian,
thus
\begin{equation*}
\begin{split}
    a^{\co}
    &~=u^{-1}(1-p)=u^{-1}au^{-1}=u^{-1}e^{-1}eau^{-1}=u^{-1}e^{-1}(eau^{-1})^{\ast}\\
    &~=(u^{\ast}eu)^{-1}a^{\ast}e=(a^{\ast}ea+ep)^{-1}a^{\ast}e.
\end{split}
\end{equation*}
Through this expression,
we naturally want to know whether $a$ is $e$-core invertible or not when there is a unique idempotent $p$ such that $(ep)^{\ast}=ep$,
$pa=0$,
$a^{\ast}ea+ep\in R^{-1}$.
However,
it is not true.
Here is a counterexample.

\begin{example}\label{eg1}
Let $R$ be an infinite matrix ring over complex filed whose rows and columns are both finite,
let conjugate transpose be the involution and $a=\sum\limits_{i=1}^{\infty}e_{i+1,i}$.
Then $a^{\ast}a=1$,
$aa^{\ast}=\sum\limits_{i=2}^{\infty}e_{i,i}$.
Set $e=1$,
$p=0$,
then $p$ is an idempotent satisfying $(ep)^{\ast}=ep$,
$pa=0$ and $a^{\ast}ea+ep=1\in R^{-1}$.
But $a$ is not group invertible,
thus $a$ is not $e$-core invertible.
\end{example}

The above counterexample shows that even if there is a unique idempotent $p$ such that $(ep)^{\ast}=ep$,
$pa=0$ and $a^{\ast}ea+ep\in R^{-1}$,
$a$ is not necessary to be $e$-core invertible in general rings.
However,
it is true when we take $R$ as a Dedekind-finite ring which satisfies the property that $ab=1$ implies $ba=1$ for any $a, b\in R$.
Let $R^{-1}_r$ and $R^{-1}_l$ denote all right invertible and left invertible elements in $R$,
respectively.
See the following result.

\begin{theorem} \label{D-finite ring}
Let $R$ be a Dedekind-finite ring and $a\in R$,
$e\in R$ is an invertible Hermitian element.
Then the following statements are equivalent: \\
$(i)$ $a\in R^{e, \co}$; \\
$(ii)$ there exists a unique idempotent $p$ such that $(ep)^{\ast}=ep$, $pa=0$, $a^{\ast}ea+ep\in R^{-1}$;\\
$(iii)$ there exists a unique idempotent $p$ such that $(ep)^{\ast}=ep$, $pa=0$, $a^{\ast}ea+ep\in R^{-1}_r$; \\
$(iv)$ there exists a unique idempotent $p$ such that $(ep)^{\ast}=ep$, $pa=0$, $a^{\ast}ea+ep\in R^{-1}_l$.\\
In this case,
\begin{equation*}
\begin{split}
    a^{e, \co}=(a^{\ast}ea+ep)^{-1}a^{\ast}e.
\end{split}
\end{equation*}
\end{theorem}

\begin{proof}
Since $a^{\ast}ea+ep$ is Hermitian,
$a^{\ast}ea+ep$ is one-sided invertible if and only if it is invertible,
hence the statements $(ii)$, $(iii)$ and $(iv)$ are equivalent.
It is sufficient to show that $(i) \Leftrightarrow (ii)$.

$(i)\Rightarrow (ii).$ Suppose that $a\in R^{e, \co}$ and let $p=1-aa^{e, \co}$,
then $a+p\in R^{-1}$ by Theorem~\ref{thm2}.
Moreover,
$e(a+p)=ea+ep\in R^{-1}$,
so $(ea+ep)^{\ast}=a^{\ast}e+ep\in R^{-1}$.
Therefore $a^{\ast}ea+ep=(a^{\ast}e+ep)(a+p)$ is invertible.

$(ii)\Rightarrow (i).$ Let $u=a+p$,
$u^{\ast}eu=a^{\ast}ea+ep\in R^{-1}$.
As $R$ is a Dedekind-finite ring,
thus $u\in R^{-1}$,
which guarantees $a\in R^{e, \co}$ and $a^{e,\co}=(a^{\ast}ea+ep)^{-1}a^{\ast}e$ by Theorem~\ref{thm2}.
\end{proof}

There is a dual result for the $f$-dual core inverse.

\begin{theorem}
Let $R$ be a Dedekind-finite ring and $a\in R$,
$f\in R$ is an invertible Hermitian element.
Then the following statements are equivalent: \\
$(i)$ $a\in R^{e, \co}$; \\
$(ii)$ there exists a unique idempotent $q$ such that $(fq)^{\ast}=fq$, $aq=0$, $af^{-1}a^{\ast}+qf^{-1}\in R^{-1}$;\\
$(iii)$ there exists a unique idempotent $q$ such that $(fq)^{\ast}=fq$, $aq=0$, $af^{-1}a^{\ast}+eqf^{-1}\in R^{-1}_r$; \\
$(iv)$ there exists a unique idempotent $q$ such that $(fq)^{\ast}=fq$, $aq=0$, $af^{-1}a^{\ast}+qf^{-1}\in R^{-1}_l$.\\
In this case,
\begin{equation*}
\begin{split}
    a^{e, \co}=f^{-1}a^{\ast}(af^{-1}a^{\ast}+qf^{-1})^{-1}.
\end{split}
\end{equation*}
\end{theorem}

Let $a\in R$ and $e, f\in R$ be two invertible Hermitian elements.
Zhu and Wang \cite{HHZ} mentioned  that $a\in R^{e, \co}\cap R_{f, \co}$ if and only if $a\in R^{\#}\cap R^{\dagger}_{e, f}$.
Mosi\'{c} et al. \cite{Mosic2} mentioned that $a\in R$ is weighted-$\mathrm{EP}$ with respect to $(e, f)$ if and only if $a\in R^{e, \co}\cap R_{f, \co}$ and $a^{e, \co}=a_{f, \co}$.
By Theorem~\ref{thm2} and Theorem~\ref{thm3},
we obtain the following theorem.

\begin{theorem} \label{EP}
Let $a\in R$ and $e, f\in R$ be two invertible Hermitian elements,
$n\geqslant 1$ is a positive integer.
The following conditions are equivalent: \\
$(i)$ $a\in R$ is weighted-$\mathrm{EP}$ with respect to $(e, f)$; \\
$(ii)$ there exists a unique idempotent $p$ such that $(ep)^{\ast}=ep$, $(fp)^{\ast}=fp$, $ap=pa=0$, $a^{n}+p\in R^{-1}$;\\
$(iii)$ there exists an element $s$ such that $(es)^{\ast}=es$, $(fs)^{\ast}=fs$, $as=sa=0$, $a^{n}+s\in R^{-1}$.\\
\end{theorem}
\begin{proof}
$(i)\Rightarrow (ii).$ Suppose $a$ is weighted-$\mathrm{EP}$ with respect to $(e, f)$,
so $a\in R^{\dagger}_{e, f}\cap R^{\#}$ and $a^{\dagger}_{e, f}=a^{\#}$.
Let $p=1-a^{\#}a=1-a^{\dagger}_{e, f}a$,
it is easily seen that $p$ is an idempotent satisfying $(ep)^{\ast}=ep$, $(fp)^{\ast}=fp$, $pa=ap=0$.
Since
\begin{equation*}
\begin{split}
    (a^{n}+1-a^{\#}a)((a^{\#})^{n}+1-a^{\#}a)=1=((a^{\#})^{n}+1-a^{\#}a)(a^{n}+1-a^{\#}a),
\end{split}
\end{equation*}
thus $a^{n}+1-a^{\#}a$ is invertible.

$(ii)\Rightarrow (iii)$ is trivial.

$(iii)\Rightarrow (i).$ According to Theorem~\ref{thm2} and Theorem~\ref{thm3},
$a\in R^{e, \co}\cap R_{f, \co}$.
It remains to show that $a^{e, \co}=a_{f, \co}$.
Write $v=a^{n}+s$.
When $n=1$,
we have $a^{e, \co}=v^{-1}av^{-1}=a_{f, \co}$.
When $n\geqslant 2$,
$a^{e, \co}=a^{n-1}v^{-1}, a_{f, \co}=v^{-1}a^{n-1}$.
From $pa=ap=0$,
we obtain $av=va$,
thus $av^{-1}=v^{-1}a$.
Therefore,
$a^{e, \co}=a^{n-1}u^{-1}=u^{-1}a^{n-1}=a_{f, \co}$.
As a consequence,
$a$ is weighted-$\mathrm{EP}$ with respect to $(e, f)$.
\end{proof}

 \bigskip

\centerline {\bf ACKNOWLEDGMENTS}
This research is supported by the National Natural Science Foundation of China (No.11871145);
NSF of Jiangsu Province (BK20200944),
Natural Science Foundation of Jiangsu Higher Education Institutions of China (20KJB110001);
the QingLan Project of Jiangsu Province.

\end{document}